\documentclass[10pt,reqno]{amsart}
\usepackage{amsmath,amscd,graphicx}
\usepackage{url}

\theoremstyle{plain}
   \newtheorem{theorem}{Theorem}[section]
   \newtheorem{proposition}[theorem]{Proposition}
   \newtheorem{lemma}[theorem]{Lemma}
   \newtheorem{corollary}[theorem]{Corollary}
   
   \newtheorem{problem}[theorem]{Problem}
\theoremstyle{definition}
   \newtheorem{definition}[theorem]{Definition}
   
\newtheorem{example}[theorem]{Example} 

\theoremstyle{remark}
 \newtheorem{remark}[theorem]{Remark}

\author[P.~Br\"and\'en]{Petter Br\"and\'en}
\thanks{Supported by the G\"oran Gustafsson Foundation.}
\address{Department of Mathematics, Royal Institute of Technology,
       SE-100 44 Stockholm, Sweden} 
\email{pbranden@math.kth.se}       
       \address{Department of Mathematics, 
Stockholm University, 
SE-106 91 Stockholm, Sweden}
\email{pbranden@math.su.se}


\newcommand{\NN}{\mathbb{N}}

\newcommand{\xx}{\mathbf{x}}
\newcommand{\yy}{\mathbf{y}}

\newcommand{\RR}{\mathbb{R}}
\newcommand{\CC}{\mathbb{C}}

\renewcommand{\Im}{{\rm Im}}

\def\newop#1{\expandafter\def\csname #1\endcsname{\mathop{\rm
#1}\nolimits}}

\newop{diag}
\newop{Arg}
\newop{supp}
\newop{per}
\newop{JP}
\newop{Sym}
\newop{St}
\newop{glb}
\newop{Trp}
\newop{Ex}

\begin{document}
\title[A generalization of the Heine--Stieltjes theorem]
{A generalization of the Heine--Stieltjes theorem }
\begin{abstract}
The Heine--Stieltjes Theorem describes the polynomial solutions, $(v,f)$ such that  $T(f)=vf$, to specific second order differential operators, $T$, with polynomial coefficients. We extend the theorem to concern all (non-degenerate) differential operators preserving the property of having only real zeros, thus solving a conjecture of B. Shapiro. The new methods  developed are used to describe intricate interlacing relations between the zeros of different pairs of solutions. This extends recent results of Bourget, McMillen and Vargas for the Heun equation and answers their question on how to generalize their results to higher degrees.  Many of the results are new even for the classical case. 
\end{abstract}
\maketitle

\section{Introduction and Main Results}
Special cases of the following general eigen-problem have been studied frequently since the 1830's. 
\begin{problem}\label{HS-problem}
Consider a linear operator $T : \CC[z] \rightarrow \CC[z]$. Describe all pairs $(v, f) \in \CC[z]\times \CC[z]$, with $f \not \equiv 0$, such that  
$T(f)= v f$. 
\end{problem}
Recently the problem has received new interest from different perspectives, see e.g. \cite{BM, BMV, MS, Shap08}. 

We say that $v$ above is a \emph{(generalized) Van Vleck polynomial} and that $f$ is a \emph{(generalized) Stieltjes polynomial} for $T$. The following theorem  is known as the Heine--Stieltjes Theorem \cite{Szego} and is due to Stieltjes \cite{Stieltjes}, except for the statement regarding the location of the zeros of the Van Vleck polynomials in (3), which is due to Van Vleck \cite{VV}.

\begin{theorem}[Heine--Stieltjes Theorem]\label{SVB} Suppose that $T : \CC[z] \rightarrow \CC[z]$ is a 
differential operator of the form  
$
T= Q_2(z)D^2+Q_1(z)D
$, 
where $D= d/dz$, $Q_2(z)=\prod_{j=1}^d(z-\alpha_j)$ and 
$$
Q_1(z) = Q_2(z)\sum_{j=1}^d \frac {\gamma_j}{z-\alpha_j}, 
$$
where $\alpha_1< \cdots < \alpha_d$ are real, $d \geq 2$, and $\gamma_1, \ldots,\gamma_d$ are positive. Let $n \geq 2$ be an integer. Then 
\begin{enumerate}
\item There are exactly $\binom {n+d-2} n$ pairs $(v,f) \in \CC[z] \times \CC[z]$ such that $T(f)=vf$ and $f$ is monic of degree $n$;
\item For each Van Vleck polynomial there is a unique monic Stieltjes polynomial;
\item The zeros of $v$ and $f$ in (1) are real, simple and belong to the interval $(\alpha_1, \alpha_d)$.
\end{enumerate}
\end{theorem}
The problem studied in this paper is how to generalize the Heine--Stieltjes Theorem to larger families of linear operator. Which properties of $T$ in the Heine--Stieltjes Theorem are essential? 
We will see that the essential property of $T$ is that it preserves the property of having only real zeros.  

A polynomial $f \in \RR[z]$ is {\em hyperbolic} if all its zeros are real and a linear operator 
$T : \RR[z] \rightarrow \RR[z]$ {\em preserves hyperbolicity} if $T(f)$ is hyperbolic or identically zero whenever $f$ is hyperbolic. Such operators were recently characterized in \cite{BBannals}.
Our extension of Theorem \ref{SVB} is to differential operators, of finite but arbitrary order, that preserve hyperbolicity.  Let 
\begin{equation}\label{diffop}
T= \sum_{k=M}^N Q_k(z)D^k, \quad \mbox{ where } Q_M(z)Q_N(z)\not \equiv 0,
\end{equation}
be a differential operator with polynomial coefficients. The number $$r= \max\{ \deg Q_k -k : M \leq k \leq N\}$$ is called the 
{\em Fuchs index} of $T$, and $T$ is {\em non-degenerate} if $\deg Q_N = N+r$. 

For a positive integer $m$, let $[m]=\{1,\ldots,m\}$. An {\em ordered partition of $[m]$ into two parts} is a pair $(A,B)$ where $A, B \subseteq [m]$, $A\cup B=[m]$ and $A\cap B=\emptyset$. 
Let $f$ and $g$ be two hyperbolic polynomials of degree $r$ and $n$, respectively and 
let  $\alpha_1 \leq \alpha_2 \leq \cdots \leq \alpha_{n+r}$ be the zeros of $fg$. Let  further $(A,B)$ be an ordered partition of 
the set $[n+r]$, such that $|A|=r$ and $|B|=n$. The pair 
$(f,g)$ is said to have {\em zero pattern} $(A,B)$ if the zeros of $f$ are 
$\{\alpha_{i} : i \in A\}$ and the zeros of $g$ are $\{\alpha_{i} : i \in B\}$. 
Note that the zero pattern is unique if and only if $f$ and $g$ are co-prime.  

 It was conjectured by B. Shapiro \cite{Shap08, Shap05}, after extensive computer experiments, that a Heine--Stieltjes Theorem should hold for non-degenerate hyperbolicity preserving differential operators. The solution to this conjecture, Theorem \ref{Main}, is our main result.  The reason why Theorem \ref{SVB} is a special case of Theorem \ref{Main} is explained in Remark \ref{why}. 
\begin{theorem}\label{Main}
Let  $T= \sum_{k=0}^N Q_k(z) D^k$,  where $Q_N(z) \not \equiv 0$,  be a non-degenerate hyperbolicity preserving differential operator with nonnegative  Fuchs index $r$, and suppose that $n\in \NN$ is such that there is no $\theta \in \RR$ such that $Q_0(\theta)= \cdots = Q_n(\theta)=0$. Then
\begin{enumerate}
\item If $f_1, f_2$ are monic polynomial for which $(v,f_1)$ and $(v,f_2)$ are solutions to Problem \ref{HS-problem}, then $f_1=f_2$;
\item For all $S \subseteq [n+r]$ of cardinality $r$ there is a unique pair $(v,f)$ of hyperbolic polynomials with zero pattern $(S,[n+r]\setminus S)$ such that $T(f)=vf$, with $f$ monic of degree $n$. These  are all Stieltjes and Van Vleck polynomials. Hence there are exactly $\binom {n+r} r$ pairs of polynomials $(v,f)$ such that 
$T(f)=vf$, with $f$ monic of degree $n$;
 \item   $v$ and $f$ are  co-prime and their zeros lie in the convex hull of the zeros of  $Q_e$, where $e= \min(n,N)$;
 \item If  there is no $\theta \in \RR$ such that $Q_1(\theta)= \cdots = Q_n(\theta)=0$ then $v$ and $f$ have simple zeros which all lie in the interior of the convex hull of the zeros of  $Q_e$, where $e= \min(n,N)$.
 \end{enumerate}
\end{theorem}
That $v$ and $f$ are co-prime in the original Heine--Stieltjes theorem  is due to Shah \cite{Shah}. 

It follows from Lemma \ref{coeff} below that the coefficients $Q_j(z)$, for $0 \leq j \leq N$, in Theorem \ref{Main} are hyperbolic. Hence the convex hull of the zeros of $Q_e(z)$ in Theorem \ref{Main} (3) and (4) is a closed interval.  

\begin{example}\label{Ex1}
A natural class of  differential operators satisfying the hypothesis of Theorem \ref{Main}
consists of operators of the form 
$$
T(f)(z)= D^{N-M}(PD^Mf)(z)= \sum_{k=M}^N \binom {N-M} {k-M} P^{(N-k)}(z)D^kf(z),
$$ where $P$ is a hyperbolic polynomial of degree $N+r$, with no zero of multiplicity larger than $N-M$. 
\end{example}

The following theorem offers a way of (in theory) generating Stieltjes and Van Vleck polynomials. Note that there is no restriction on the common zeros of the $Q_k$s. 
\begin{theorem}\label{Main2}
Let  $T= \sum_{k=0}^N Q_k(z) D^k$ be a non-degenerate hyperbolicity preserving differential operator with nonnegative  Fuchs index $r$. Suppose that $n \in \NN$ is such that $T(z^n) \not \equiv 0$ and  that $A \subseteq [n+r]$ is a set of cardinality $r$. Let $a$ be the smallest zero of $Q_e(z)$, where $e= \min(n,N)$. 

Define two sequences of polynomial $\{f_i\}_{i=0}^\infty$ and $\{v_i\}_{i=1}^\infty$ recursively by 
$f_0(z)= (z-a)^n$ and 
$$
T(f_i)=v_{i+1}f_{i+1},
$$
where $v_{i+1}f_{i+1}$ is the unique factorization of $T(f_i)$ such that $f_{i+1}$ is monic and 
the pair $(v_{i+1}, f_{i+1})$ has zero pattern $(A,B)$, where $B= [n+r] \setminus A$. Then 
$$
\lim_{i \rightarrow \infty} f_i =f \quad \mbox{ and } \quad \lim_{i \rightarrow \infty} v_i =v,
$$
where $T(f)=vf$ and $(v, f)$ have zero pattern $(A,B)$.

\end{theorem}

In Section \ref{inter} we use Theorem \ref{Main2} to derive interlacing properties of the zeros of different Stieltjes and Van Vleck polynomials. These generalize recent results of  Bourget, McMillen and Vargas \cite{BMV} and Bourget and McMillen \cite{BM} derived for the Heun equations, and answers the question raised in \cite{BM} on how to generalize their results to 
Van Vleck polynomials of higher degree. 

\section{Properties of Hyperbolicity Preservers}
Here we collect the results on hyperbolicity preserving differential operators  needed to prove Theorem \ref{Main}. Some results are new and of independent interest. 
A polynomial $f(z_1,\ldots, z_n) \in \CC[z_1,\ldots, z_n]$ is {\em stable} if 
$f(z_1,\ldots, z_n) \neq 0$ whenever $\Im(z_j)>0$ for all $1\leq j \leq n$. Hence a univariate polynomial with real coefficients is stable if and only it is hyperbolic.  
The following characterization of hyperbolicity preservers in the Weyl algebra was given in \cite{BBlondon}. 
\begin{theorem}\label{char}
Let $T= \sum_{k=M}^N Q_k(z)D^k$, where $Q_k(z) \in \RR[z]$ for all $M \leq k \leq N$ and   $Q_M(z)Q_N(z) \not \equiv 0$. Let further $G_T(z,w)= \sum_{k=M}^N Q_k(z)w^{N-k} \in \RR[z,w]$. The following are equivalent: 
\begin{enumerate}
\item $T$ preserves hyperbolicity;
\item $G_T(z,w)$ is stable; 
\item $G_T(z,w)$ can be expressed as 
$$
G_T(z,w)= \pm \det(zA+wB+C)
$$
where $A$ and $B$ are positive semi-definite symmetric matrices and $C$ is symmetric.
\end{enumerate}
\end{theorem} 

\begin{remark}
Using Theorem \ref{char} one can easily construct linear operators $T$ for which Theorem \ref{Main} is applicable: Choose $A, B, C$ to be symmetric $(N+r)\times (N+r)$ matrices satisfying:
 $A$ is positive definite, $B$ is positive semi-definite and has $0$ as an eigen-value of multiplicity $r$, and 
 $C$ is such that $\det(\theta A + Bw +C) \not \equiv 0$ for all $\theta \in \RR$. Then choose 
 $T$ such that $G_T(z,w)= \pm \det(zA+wB+C)$. 
 \end{remark} 

\begin{remark}\label{why}
Let $T$ be as in Theorem \ref{SVB}. Suppose that $G_T(z,w)=Q_2(z)+Q_1(z)w =0$ 
where $\Im(z)>0$. Since $Q_2$ is hyperbolic and 
$$
\Im\left(\frac 1 {z-\alpha_j}\right)= - \frac {\Im(z)}{|z-\alpha_j|^2}<0
$$
we have $w\neq 0$ and 
$$
\Im(w^{-1})= \sum_{j=1}^d \gamma_j \frac {\Im(z)}{|z-\alpha_j|^2} >0.
$$
Hence $\Im(w)<0$. Thus $G_T(z,w)$ is stable and by Theorem \ref{char}, $T$ preserves hyperbolicity. Note also that $Q_1$ and $Q_2$ are co-prime. Therefore Theorem \ref{SVB} follows from Theorem \ref{Main}. In fact, the linear operators considered in Theorem \ref{SVB} are precisely those with $M=1$ and $N=2$ in Theorem \ref{Main}. 
\end{remark}

\begin{proposition}\label{posmono}
Let $T= \sum_{k=M}^N Q_k(z)D^k$, where $Q_k(z) \in \RR[z]$ for all $M \leq k \leq N$ and   $Q_M(z)Q_N(z) \not \equiv 0$,  be a hyperbolicity preserver. Let $r= \deg Q_M -M$. Then 
\begin{enumerate}
\item $\deg Q_i \leq r+ i$ for all $M \leq i \leq N$;
\item If $a_i$ is the coefficient in front of $z^{r + i}$ in $Q_i(z)$, then all the zeros of the polynomial 
$
\sum_{i=M}^Na_iz^i 
$
are real and nonpositive. In particular, the nonzero $a_i$s have the same sign, and the indices $i$ such that $a_i\neq 0$ form an interval; 
\item  $T(z^n)$ is of degree $n +r$ for all $n \geq M$ and the leading coefficient of $T(z^n)$ is equal to $\lambda_n = \sum_i i!a_i\binom n i$. In particular, if $a_M>0$, then 
$$
\lambda_{n+1}-\lambda_n = \sum_i i!a_i\binom n {i-1}>0
$$
for all $n \geq M$, unless $M=N=0$. 
\end{enumerate}
\end{proposition}
\begin{proof}
Property (1) is a direct consequence of Theorem \ref{char} and the fact that the support of a stable polynomial forms a jump system, see \cite[Theorem 3.2]{BraHPP}. 

Let $\gamma >0$. Then $\gamma^{-r-N}G_T(\gamma z,\gamma w)$ is stable by Theorem \ref{char}. By Hurwitz' theorem\footnote{See \cite[p.~96]{COSW} for an appropriate multivariate version.} so is the homogeneous polynomial 
$$
g(z,w)=\lim_{\gamma \rightarrow \infty} \gamma^{-r-N}G_T(\gamma z,\gamma w)=  \sum_{i=M}^Na_iz^{r+i}w^{N-i}. 
$$
A property of homogeneous stable polynomials is that all nonzero coefficients have the same phase, see \cite[Theorem 6.1]{COSW}.
The polynomial $g(z, 1+i\epsilon)$ is stable if $\epsilon>0$. Hence, by Hurwitz' theorem, $g(z,1)$ is stable and hence hyperbolic.  Since the nonzero coefficients of $g(z,1)$ have the same sign,  the zeros of $g(z,1)$ are nonpositive. The indices $i$ such that $a_i\neq 0$ of such a polynomial form an interval. 

Property (3) follows easily from (1) and (2). 
\end{proof}

Let $\alpha_1 \leq \alpha_2 \leq \cdots \leq \alpha_n$ and 
$\beta_1 \leq \beta_2 \leq \cdots \leq \beta_m$ be the zeros (counted with
multiplicities) of two hyperbolic polynomials $f$ and $g$,  with $\deg f=n$ and $\deg g=m$. We say that these zeros 
{\em interlace} if they can 
be ordered so that either   $\alpha_1 \leq \beta_1 \leq \alpha_2 \leq \beta_2 \leq 
\cdots$ or $\beta_1 \leq \alpha_1 \leq \beta_2 \leq \alpha_2 \leq 
\cdots$, in which case  $|n-m|\le 1$. Note that by our convention, the zeros of any two polynomials 
of degree $0$ or $1$ interlace. It is not difficult to show that if   
the zeros of $f$ and $g$ interlace then 
the {\em Wronskian} 
$W[f,g]:=f'g-fg'$ is either nonnegative or nonpositive on the whole real axis 
$\RR$, see, e.g., \cite{RS}. In the case that  $W[f,g] \leq 0$ we say that $f$ and $g$ 
are  in  {\em proper position}, denoted $f \ll g$. 
For technical reasons we also say that the zeros of the polynomial $0$ 
interlace the zeros of any (nonzero) hyperbolic polynomial and write 
$0 \ll f$ and $f \ll 0$.
Note 
that if $f$ and $g$ are (nonzero) hyperbolic polynomials such that $f \ll g$ and $g \ll f$ then $f$ and $g$ must be constant multiples 
of each other, that is, $W[f,g]\equiv 0$. 

The following theorem is a version of the classical Hermite-Biehler 
theorem, see \cite{RS}.

\begin{theorem}[Hermite-Biehler theorem]
Let $h := f +ig \in \CC[z]$, where $f,g \in \RR[z]$. Then 
$h$ is stable if and only if $g \ll f$. 
\end{theorem}  
The next proposition follows from the fact that a hyperbolicity preserving differential operator also preserves stability when acting on complex polynomials, see \cite{BBlondon}.

\begin{proposition}\label{properpres}
Let $T$ be a hyperbolicity preserving differential operator. Suppose that $f,g \in \RR[z]$ are such that 
$f \ll g$. Then $T(f) \ll T(g)$, unless $T(f)=T(g) \equiv 0$. 
\end{proposition}

\begin{lemma}\label{coeff}
Let $T= \sum_{k=M}^N Q_k(z)D^k$, where $Q_k(z) \in \RR[z]$ for all $M \leq k \leq N$ and   $Q_M(z)Q_N(z) \not \equiv 0$, be a hyperbolicity preserver. Then  $Q_j \ll Q_{j+1}$ for all $M \leq j \leq N-1$. In particular $Q_j$ is hyperbolic or identically zero for all $M \leq j \leq N$. 
\end{lemma}
\begin{proof}
Let $S$ be the linear operator on $\RR[z,w]$ defined by $S(z^kw^\ell)= z^kw^\ell$ if $\ell \in \{N-j, N-j+1\}$ and 
$S(z^kw^\ell)=0$ otherwise. It follows from the theorem characterizing multidimensional multiplier sequences in \cite{BBlondon} that $S$ preserves stability. The polynomials $G_T(z,w)$ in Theorem \ref{char} is stable and 
$$
S(G_T)= w^{N-j}(Q_j(z)+wQ_{j-1}(z)).
$$
Hence $Q_j(z)+wQ_{j-1}(z)$ is stable and so is  the polynomial $Q_j(z)+iQ_{j-1}(z)$. The lemma now follows  from the Hermite-Biehler theorem. 
\end{proof}

If $f$ is a hyperbolic polynomial we let $I(f)$ denote the smallest interval containing all the zeros of $f$. 
\begin{lemma}\label{inclusion}
Let $T=\sum_{k= M}^NQ_k(z)D^k$ be a non-degenerate hyperbolicity preserver, where $Q_MQ_N \not \equiv 0$.  Suppose that  $f$ is  a hyperbolic polynomial of degree $n \geq M$. Then 
$$
I(T(f)) \subseteq I(Q_ef),
$$
where $e=\min(n,N)$.
\end{lemma}
\begin{proof}
 Since $T$ is non-degenerate we know by Proposition \ref{posmono} that 
$\deg Q_k = \deg Q_M + k$, for all $M\leq k \leq N$. By Proposition \ref{posmono} the sign of the leading coefficients of 
the $Q_k$'s is the same for all $M \leq k \leq N$.  Since  $Q_M \ll Q_{M+1} \ll \cdots \ll Q_N$, by Lemma \ref{coeff}, it follows that 
  $I(Q_i f^{(i)}) \subseteq I(Q_e f)$ for all $M \leq i \leq n$. Hence, 
all the polynomials $Q_i f^{(i)}$ are nonzero and have the same sign outside 
$I(Q_e f)$. Since $T(f)$ is a nonnegative sum of these polynomials we have 
$I(T(f)) \subseteq I(Q_e f)$. 
\end{proof}

The next lemma is often referred to as Obreschkoff's theorem, for a proof see \cite{RS} or \cite{BBinvent}. 
\begin{lemma}\label{obr}
Let $f,g \in \RR[z]$. Then all nonzero polynomials in the space 
$$
\{\alpha f +\beta g: \alpha, \beta \in \RR\}
$$
are hyperbolic if and only if $f \ll g$, $g \ll f$ or $f=g=0$. 
\end{lemma}

 \begin{lemma}\label{RHP}
Let $F$ be a hyperbolic polynomial of degree $d$ and let $f$ be a hyperbolic 
polynomial. Suppose that $(z-\theta)^{N}\mid F(D)f$, where $N\geq 2$. Then 
$(z-\theta)^{N+d}\mid f$ or $F(D)f \equiv 0$. 
\end{lemma}
\begin{proof}
The lemma follows from the case $d=1$ by induction on $d$.  Hence let $F=a+z$ where 
$a \in \RR$.  It follows from e.g. Lemma \ref{obr} that the zeros of $f$ and the zeros of $f'$ interlace the zeros of $af + f'$. Hence $(z-\theta)^{N-1}\mid f'$ and 
$(z-\theta)^{N-1}\mid f$, and since $N \geq 2$ we have $(z-\theta)^{N}\mid f$. Since 
$f'=F(D)f-af$, $(z-\theta)^N \mid F(D)f$ and $(z-\theta)^{N}\mid f$ we have that $(z-\theta)^{N}\mid f'$. It follows that $(z-\theta)^{N+1}\mid f$ as was to be 
shown.
%
\end{proof}

Lemma \ref{RHP} generalizes to differential operators as follows.
\begin{proposition}\label{multi}
Let $T=\sum_{k=M}^{N}Q_{k}(z)D^{k}$, where $Q_{M}Q_{N} \not \equiv 0$, be a hyperbolicity preserver  and 
let $f$ be a hyperbolic polynomial with $n= \deg f \geq M$. Suppose 
that $(z-\theta)^{d}\mid T(f)$, where $d \geq 2$. Then either
$Q_{M}(\theta)=Q_{M+1}(\theta)=\cdots=Q_{n}(\theta)=0$ or 
$$
(z-\theta)^{d+t}\mid f(z),
$$ 
where $t=\max \{ k \leq N : Q_{k}(\theta)\neq 0\}$. 
\end{proposition}
\begin{proof}
Let $f$ and $\theta$ be as in the statement of the proposition and let 
$$
G(\theta,w)= \sum_{k=M}^{N}Q_k(\theta)f^{(k)}(w +\theta)=\sum_{k=0}^n\frac {T(f^{(k)})(\theta)}{k!}w^k.
$$
If $Q_{M}(\theta)=Q_{M+1}(\theta)=\ldots=Q_{n}(\theta)=0$ there is nothing to prove 
so assume otherwise. It follows from Theorem \ref{char} and the fact that stability is preserved under setting variables to real values (see e.g.~\cite{BBinvent}) that the polynomial $F(z)=\sum_{k=M}^rQ_k(\theta)z^k$ is hyperbolic. By hypothesis  $F(D)[f(w+\theta)]=G(\theta,w) \not \equiv 0$. Clearly the zeros of 
$f^{(i)}$ and $f^{(i+1)}$ interlace for all $0 \leq i \leq n$. By Lemma \ref{properpres}  the zeros 
of $T(f^{(i)})$ and $T(f^{(i+1)})$ interlace for all $0 \leq i \leq n-M$.  Since $(z-\theta)^d \mid T(f)$ it follows that 
$(z-\theta)^{d-i} \mid T(f^{(i)})$ for all $0 \leq i \leq d$. Hence $w^d \mid G(\theta,w)$ which by 
Lemma \ref{RHP} implies $w^{d+t} \mid f(w+\theta)$, that is  $(z-\theta)^{d+t}\mid f(z)$. 
\end{proof}
%


\section{Proofs of the Main Results}
We will make use of the following generalization \cite{Shap08} of Heine's theorem \cite{Heine}. 
\begin{theorem}[B. Shapiro]\label{sapiro}
Let $T$ be a non-degenerate differential operator as in \eqref{diffop}, and let $n \in \NN$. Suppose that $T$ has Fuchs index  $r \geq 0$. Then there are at most $\binom {n+r} r$ distinct Van Vleck polynomials $v$, for which $T(f)=vf$ for some $f \in \CC[z]$ of degree $n$. 
\end{theorem}

\begin{lemma}\label{vv}
Let $T=\sum_{k=M}^{N}Q_{k}(z)D^{k}$, where $Q_{M}Q_{N} \neq 0$ and $M<N$, be a hyperbolicity preserver with Fuchs index $r \geq 0$. Then there is exactly one monic Stieltjes polynomial to each non-zero Van Vleck polynomial.
\end{lemma}
\begin{proof}
The set of Stieltjes polynomials (union the zero polynomial) corresponding to the same Van Vleck polynomial, $v$,  
is a linear space. Hence if there were two monic Stieltjes polynomials corresponding to $v$, then there would be two Stieltjes polynomials of different degrees $m>n$, also corresponding to $v$. However, by Proposition \ref{posmono}, the leading coefficient of a Van Vleck polynomial pairing with a Stieltjes polynomial of degree $k$ is $\lambda_k$, and since $\{\lambda_i\}_{i=M}^\infty$ is strictly increasing we have a contradiction.
\end{proof}

\begin{proof}[Proof of Theorem \ref{Main}]
Let $T$ and $n$ be as in the statement of the theorem and let $S \subseteq [n+r]$ be of cardinality $r$. We first prove that there is a pair $(v,f)$ of hyperbolic polynomials with zero pattern $(S,[n+r]\setminus S)$ such that $T(f)=vf$ with $f$ of degree $n$. Let $I$ be the convex hull of the zeros of $Q_e$. Hence 
$I$ is a closed bounded interval, possibly without interior.  Define a function $\phi : I^n \rightarrow I^n$ as follows. If $\alpha=(\alpha_1, \ldots, \alpha_n) \in I^n$ form the polynomial $h(z)= \prod_{j=1}^n (z-\alpha_j)$. By Proposition \ref{posmono}, $T(h)$ is hyperbolic of degree $n+r$. Hence we may write   $T(h)$ uniquely as $v g$ where  $g$ is monic and $(v,g)$ has zero pattern $(S,[n+r]\setminus S)$. Let $\phi(\alpha) = (\beta_1,\ldots, \beta_n)$, where 
$\beta_1 \leq \cdots \leq \beta_n$ are the zeros of $g$. By Hurwitz' theorem $\phi$ is continuous, which by Brouwer's fixed point theorem implies that $\phi(\alpha) = \alpha$ for some 
$\alpha \in I^n$. Hence $T(f)=vf$ for some $(v,f)$ with zero pattern $(S,[n+r]\setminus S)$.

Next we prove that if  $(v,f)$ is a pair of hyperbolic polynomials with zero pattern $(S,[n+r]\setminus S)$ such that $T(f)=vf$ with $f$ of degree $n$, then $v$ and $f$ are co-prime. Assume otherwise that $f(\theta)=v(\theta)=0$ and that the multiplicity of $\theta$ in $f$ is $m$. Then $(z-\theta)^{m+1} \mid T(f)$ which by Proposition \ref{multi} implies $(z-\theta)^{m+1} \mid f$, which is a contradiction. Note that this means that the zero pattern of $(v,f)$ is unique. 

(1) is just a special case of Lemma \ref{vv}. 

From what we have proved so far we deduce that there are at least $\binom {n+r} r$ distinct Van Vleck polynomials. By Theorem \ref{sapiro} these are all Van Vleck polynomials, and by the above there is a unique Stieltjes polynomial to each Van Vleck polynomial. 

It remains to prove (4) for a given pair $(v,f)$ of Van Vleck and Stieltjes polynomials.  Suppose that 
$f$ has a zero $\theta$ of multiplicity $m>1$. Then $(z-\theta)^m \mid T(f)$. 
By Proposition \ref{multi}, $(z-\theta)^{m+1} \mid f$, which is a contradiction. If $v$ has has a zero $\theta$ of multiplicity $m>1$, then $(z-\theta)^m \mid T(f)$.  Hence $(z-\theta)^{m+1} \mid f$ by Proposition \ref{multi}. This contradicts the simplicity of the zeros of $f$ just derived. Let $a$ be an endpoint of $I$. Since all zeros of $f$ are simple and located in $I$ we know that $f'(a)\cdots f^{(n)}(a) \neq 0$. As in the proof of Lemma \ref{inclusion} we note that the signs of all nonzero $Q_i(a)f^{(i)}(a)$, for $0 \leq i \leq e$,  are the same. By hypothesis $Q_i(a)f^{(i)}(a) \neq 0$ for some $i$, which proves that $T(f)(a) \neq 0$. 
\end{proof}
\begin{proof}[Proof of Theorem \ref{Main2}]
Let $I$ and $\phi : I^n \rightarrow I^n$ be as in the proof of Theorem \ref{Main}. Let further $\leq$ denote the usual product  partial order on $\RR^n$. We claim that if $\alpha, \beta \in \RR^n$ are weakly increasing, i.e., 
$\alpha_1 \leq \cdots \leq \alpha_n$,  $\beta_1 \leq \cdots \leq \beta_n$ and $\alpha \leq \beta$, then 
$\phi(\alpha) \leq \phi(\beta)$.  To see this write $\alpha \ll \beta$ if 
$\alpha_1 \leq \beta_1 \leq \alpha_2 \leq \beta_2 \leq \cdots \leq \alpha_n \leq \beta_n$. Clearly 
$\alpha \leq \beta$ whenever $\alpha \ll \beta$. 
By 
Proposition \ref{properpres} it follows that $\phi(\alpha) \ll \phi(\beta)$ whenever $\alpha, \beta \in \RR^n$ are weakly increasing and $\alpha \ll \beta$. However, if $\alpha \leq \beta$  are weakly increasing then there is a sequence, $\{\gamma^k\}_{k=0}^m \subset \RR^n$, such that $\alpha = \gamma^0 \ll \gamma^1 \ll \cdots \ll \gamma^m = \beta$. Hence 
$\phi(\alpha) \ll \phi(\gamma^1) \ll \cdots \ll \phi(\beta)$, and thus $\phi(\alpha) \leq \phi(\beta)$ which proves the claim. 

Define a sequence $\{\alpha^k\}_{k=0}^\infty \subset I^n$ recursively by $\alpha^0=(a,\ldots, a)$, and 
$\alpha^{k+1}= \phi(\alpha^k)$ for all $k \in \NN$. Thus the coordinates of $\alpha^k$ are the zeros of $f_k$.  Since $\alpha^0$ is the unique smallest element in $I^n$ with respect 
to $\leq$, and $\phi : I^n \rightarrow I^n$, we have $\alpha^0 \leq \alpha^1$. Since $\phi$ preserves the partial order $\leq$ we have $\alpha^0 \leq \alpha^2 \leq \alpha^3 \leq \cdots $. By compactness, the sequence $\{\alpha^k\}_{k=0}^\infty$ converges to a limit $\alpha=(\alpha_1, \ldots, \alpha_n) \in I^n$. It follows that $\alpha$ is  a fixed point of $\phi$ and that 
$\lim_{i \rightarrow \infty} f_i = f:= \prod_{i=1}^n (z-\alpha_i)$. Since all the $v_i$s have the same leading 
coefficient, $\lambda_n$, it follows that $\lim_{i \rightarrow \infty} v_i = v$ and that 
$T(f)=vf$ where $(v,f)$ has zero-pattern $(A,B)$. 

\end{proof}



\section{Interlacing Properties}\label{inter}
In this section we will prove interlacing properties of zeros of  Stieltjes polynomial and of Van Vleck polynomials. The first two theorems concern Stieltjes polynomials of the same or consecutive degrees. These theorems generalize the theorems obtained by Bourget and McMillen \cite{BM}  for the Heun equation. In \cite{BM} it was asked if and how their results could be extended to the case of Fuchs index greater than one. Theorems \ref{sames} and \ref{cons} answers this. Our  theorems concerning Van Vleck polynomials generalize to arbitrary nonnegative  Fuchs index the results of Bourget, McMillen and Vargas \cite{BMV} regarding the Heun equation.

\begin{definition}\label{def}
Let $\binom {[n+r]} n$ denote the set of all $n$ element subsets of $[n+r]$, and 
let $T=\sum_{k=M}^{N}Q_{k}(z)D^{k}$, where $Q_{M}Q_{N} \not \equiv 0$, be a non-degenerate hyperbolicity preserver with nonnegative Fuchs index $r$. Let further 
$B \in \binom {[n+r]} n$ and $A = [n+r]\setminus B$, where $n \geq M$. We denote by $(v_{A,n}, f_{B,n})$ the pair of polynomials with zero-pattern $(A,B)$ afforded by Theorem \ref{Main2}. That is, $f_{B,n}$ is monic of degree $n$ and  $T(f_{B,n})=v_{A,n}f_{B,n}$. 
\end{definition}
Henceforth it will tacitly be assumed that $n \geq M$ as in Definition \ref{def}.

A pair $(\xx, \yy)$ of real vectors $\xx = (x_1, \ldots, x_{k})$ and $\yy = (y_1, \ldots, y_{k})$ are said to be in {\em proper position}, written $\xx \ll \yy$, if 
$$x_1 \leq y_1 \leq x_2 \leq y_2 \leq \cdots \leq x_k \leq y_k.$$ 
 If $A$ is a set of integers let 
 $[A]$ denote the vector obtained by ordering its elements increasingly. 
  Define a relation ${\rightarrow}$ on $\binom {[n+r]} n$ by:  $A \rightarrow B$ if 
 $2[A] \ll 2[B]+\mathbf{1}$, where $\mathbf{1}=(1,\ldots,1)$. 
 
 The relation $\rightarrow$ may be described differently as follows. 
  If $\xx =(x_1, \ldots, x_{n+r}) \in \RR^{n+r}$  and $B \in \binom {[n+r]} n$,  let $\xx_B$ be the vector obtained by deleting the coordinates labeled by $[n+r]\setminus B$. Then 
$B \rightarrow C$ if for all 
$\xx, \yy \in \RR^{n+r}$ with $\xx \ll \yy$ we have $\xx_B \ll \yy_C$. 

\begin{lemma}\label{folklore}
Suppose that the polynomials $f$ and $g$ have interlacing zeros and that $\theta \in \RR$ is a common zero of $f$ and $g$, of multiplicity one in each. Then $\theta$ is a zero of multiplicity exactly two of the polynomial $g'(\theta)f(z)-f'(\theta)g(z)$, unless $g'(\theta)f(z)-f'(\theta)g(z) \equiv 0$. 
\end{lemma}
\begin{proof}
Let $F(z)= g'(\theta)f(z)-f'(\theta)g(z)$ and suppose that $F(z) \not \equiv 0$. Clearly  $F(\theta)=F'(\theta)=0$. By Lemma \ref{obr}, the zeros of $F(z)$ and $f(z)$ are interlacing. Hence $\theta$ is a zero of $F(z)$ of multiplicity at most two. 
\end{proof}

The following theorem describes interlacing relationships betwen Stieltjes polynomials of the same degree. 
\begin{theorem}\label{sames}
Let $T$ be a non-degenerate hyperbolicity preserver with nonnegative Fuchs index $r$ and let 
$B,C \in \binom {[n+r]} n$. If $B \rightarrow C$ then $f_{B,n} \ll f_{C,n}$. 

Moreover, if there is no $\theta \in \RR$ for which $Q_1(\theta) = \cdots = Q_n(\theta)=0$, then 
$f_{B,n}$ and  $f_{C,n}$ are co-prime, i.e., their zeros strictly interlace. 
\end{theorem}
\begin{proof}
Let $\{f_{B,n}^i\}_{i=0}^\infty$ and $\{f_{C,n}^i\}_{i=0}^\infty$ be the sequences of polynomials constructed as in Theorem \ref{Main2} for $f_{B,n}$ and $f_{C,n}$, respectively. We prove by induction that $f_{B,n}^i \ll f_{C,n}^i$ for all 
integers $i \geq 0$. The induction start is trivial since $f_{B,n}^0 = f_{C,n}^0$. 

Suppose that $f_{B,n}^i \ll f_{C,n}^i$. Then, by Proposition \ref{properpres}, we have $T(f_{B,n}^i) \ll T(f_{C,n}^i)$. The polynomials $f_{B,n}^{i+1}$ and $f_{C,n}^{i+1}$ are obtained by keeping the zeros of  $T(f_{B,n}^i)$ and $T(f_{C,n}^i)$ that are labelled by $B$ and $C$, respectively. Hence, by the (alternative) definition of $\rightarrow$ we have $f_{B,n}^{i+1} \ll f_{C,n}^{i+1}$. 

The first part of the theorem now follows by letting $i \rightarrow \infty$. 

To prove the second part assume that $f_{B,n}(\gamma)= f_{C,n}(\gamma)=0$ for some $\gamma \in \RR$. By Theorem \ref{Main}, $f_{B,n}$ and $v_{B',n}$ are simple-rooted and co-prime. The same holds for $f_{C,n}$ and $v_{C',n}$. By Lemma \ref{folklore}, $\gamma$ is a zero of multiplicity exactly two of the polynomial $T(f_{B,n} - a f_{C,n})=v_{B',n}f_{B,n} - a v_{C',n}f_{C,n}$, where 
$$
a= \frac { (v_{B',n}f_{B,n})'(\gamma)}{ (v_{C',n}f_{C,n})'(\gamma)}.
$$
But then, by Proposition \ref{multi},  $\gamma$ is a zero of multiplicity at least three of the polynomial $f_{B,n} - a f_{C,n}$, which contradicts Lemma \ref{folklore}. 
\end{proof}
\begin{remark}
Given $B=\{b_1< \cdots < b_n\} \subseteq [n+r]$ there are exactly $b_1(b_2-b_1)\cdots (b_n-b_{n-1})$ sets 
$A \in \binom {[n+r]} n$ for which $A \rightarrow B$. Indeed, to construct $A$ we pick an integer from the interval $[b_k+1, b_{k+1}]$, independently for each $k$. By a counting argument there are exactly $\binom {2n+r} r$ relations (including the trivial ones) in $\binom {[n+r]} n$.
\end{remark}

Similarly, a pair $(\xx, \yy)$ of real vectors $\xx = (x_1, \ldots, x_{k})$ and $\yy = (y_1, \ldots, y_{k+1})$ are said to be in {\em proper position}, written $\xx \ll \yy$, if 
$$ y_1 \leq x_1 \leq y_2 \leq x_2 \leq \cdots \leq x_k \leq y_{k+1}.$$ 
The relation $\rightarrow$ extends naturally to concern $\binom {[n+r]} n \times \binom {[n+1+r]} {n+1}$: $A \rightarrow B$ if 
 $2[A] \ll 2[B]+\mathbf{1}$. Again, $A \rightarrow B$ if and only if 
  for  all $\xx \in \RR^{n+r}$ and $\yy \in \RR^{n+r+1}$ with $\xx \ll \yy$ we have $\xx_A \ll \yy_B$. 

The next theorem describes interlacing relationships between Stieltjes polynomials of consecutive degrees. Its proof is very similar to that of Theorem \ref{sames} and is  therefore omitted. 
\begin{theorem}\label{cons}
Let $T$ be a non-degenerate hyperbolicity preserver with nonnegative Fuchs index $r$ and let 
$B \in \binom {[n+r]} n$ and $C \in \binom {[n+1+r]} {n+1}$. If $B \rightarrow C$ then $f_{B,n} \ll f_{C,n+1}$. 

Moreover, if there is no $\theta \in \RR$ for which $Q_1(\theta) = \cdots = Q_n(\theta)=0$, then 
$f_{B,n}$ and  $f_{C,n}$ are co-prime, i.e., their zeros strictly interlace. 
\end{theorem}

We have the following theorem describing interlacing properties of Van Vleck polynomials that have  Stieltjes polynomials of  the same degree. 

\begin{theorem}\label{samev}
Let $T$ be a non-degenerate hyperbolicity preserver with nonnegative Fuchs index $r$ and let 
$A \in \binom {[n+r-1]} r$. Then $v_{A,n} \ll v_{A+\mathbf{1},n}$, where $A+\mathbf{1}:= \{a+1, a\in A\}$. 

Moreover, if there is no $\theta \in \RR$ for which $Q_1(\theta) = \cdots = Q_n(\theta)=0$, then 
$v_{A,n}$ and  $v_{A+\mathbf{1},n}$ are co-prime, i.e., their zeros strictly interlace. 
\end{theorem}
\begin{proof}
Clearly $(A+\mathbf{1})' \rightarrow A'$, where $(A+\mathbf{1})'=[n+r]\setminus (A+\mathbf{1})$ and $A' = [n+r]\setminus A$. By Theorem \ref{sames}, $f_{(A+\mathbf{1})' } \ll f_{A' }$ and by 
Proposition \ref{properpres}
$$
v_{A+\mathbf{1}} f_{(A+\mathbf{1})' }= T(f_{(A+\mathbf{1})' }) \ll T(f_{A' }) = v_Af_{A' }.
$$
Let $x_1 \leq \cdots \leq x_{n+r}$ and $y_1 \leq \cdots \leq y_{n+r}$ be the zeros of $v_{A+\mathbf{1}} f_{(A+\mathbf{1})' }$ and $v_Af_{A' }$, respectively.  Then, for $a \in A$, 
$$
x_1 \leq y_1 \leq \cdots \leq x_a \leq y_a \leq x_{a+1} \leq y_{a+1} \leq \cdots \leq x_{n+r} \leq y_{n+r}. 
$$  
Since the zeros of $v_A$ are $\{ y_a : a \in A\}$ and the zeros of $v_{A+\mathbf{1}}$ are $\{ x_{a+1}: a \in A\}$, it follows that $v_A \ll v_{A+\mathbf{1}}$. 

The proof of the strict interlacing of the zeros follows just as in the proof of Theorem \ref{sames}. 
\end{proof}

The next theorem describes interlacing properties of  Van Vleck polynomials that have Stieltjes polynomials of consecutive degrees. The proof is very similar to that of Theorem \ref{samev} and is therefore omitted.

\begin{theorem}\label{conv}
Let $T$ be a non-degenerate hyperbolicity preserver with nonnegative Fuchs index $r$ and let 
$A \in \binom {[n+r]} {r}$. Then $v_{A,n+1}\ll v_{A,n} \ll v_{A+\mathbf{1},n+1}$.

Moreover, if there is no $\theta \in \RR$ for which $Q_1(\theta) = \cdots = Q_n(\theta)=0$, then 
the interlacing of the zeros is strict.
\end{theorem}

Suppose that $T=\sum_{k=M}^N Q_k(z) D^k$ is a non-degenerate hyperbolicity preserver with Fuchs index  $r=1$, and that $Q_M(\theta)=\cdots = Q_n(\theta)$ for  no $\theta \in \RR$. Then the Van Vleck polynomials are of degree one and there are exactly $n+1$ of them. Define the $n$th \emph{spectral polynomial}, $p_n(z)$, to be the monic polynomial whose zeros are precisely the zeros of the $n+1$ Van Vleck polynomials. 

\begin{corollary}\label{vvv}
Suppose that $T=\sum_{k=M}^N Q_k(z) D^k$ is a non-degenerate hyperbolicity preserver with Fuchs index  $r=1$, and that $Q_M(\theta)=\cdots = Q_n(\theta)$ for no $\theta \in \RR$. Then 
the zeros of $p_n$ and $p_{n+1}$ interlace. Moreover $p_n$ and $p_{n+1}$ have no zeros in common. 
\end{corollary}
\begin{proof}
The interlacing property follows immediately from Theorem \ref{conv}. That $p_n$ and $p_{n+1}$ have no zeros in common is a consequence of Lemma \ref{vv}.
\end{proof}
For the Heun equation Corollary \ref{vvv} reduces to the main result, Theorem 2.1, of \cite{BMV}. 

\section{Open Questions and Further Directions}
The Jacobi orthogonal polynomials arise as the Stieltjes polynomials corresponding to the case when $Q_2(z)$ is of degree $2$ in the classical Heine--Stieltjes theorem. Given a hyperbolicity preserving differential operator as in Theorem \ref{Main} (4) of Fuchs index $r=0$, it is natural 
to ask whether the Stieltjes polynomials form an orthogonal family. This was suggested by T. McMillen.

The linear operators considered in this paper are all finite order differential operators. One may 
try to extend our results to hyperbolicity preservers that are not finite order differential operators.  There is a natural class of such operators for which there are pairs of Van Vleck and Stieltjes polynomials for each possible zero pattern. The problem is to ensure the uniqueness of 
the Stieltjes--Van Vleck pair given the zero pattern. \\[2ex]

\noindent 
\textbf{Acknowledgments.} I thank Julius Borcea and Boris Shapiro for introducing me to the Heine--Stieltjes problem, and Julius Borcea for bringing \cite{BM} to my attention.

\bibliographystyle{amsplain}
\bibliography{HS}

\end{document}